\documentclass[reqno]{amsart}
\usepackage{amsbsy,amssymb,amscd,amsfonts,latexsym,amstext,delarray,
amsmath,mathrsfs,epsfig,graphicx,dirtytalk}
\newenvironment{nitemize}
{\begin{list}{$\bullet$} 
{\setlength{\parsep}{1ex}
\setlength{\topsep}{1.1ex}
\setlength{\partopsep}{0ex}
\setlength{\labelwidth}{0.5cm}
\setlength{\itemindent}{0cm}
\setlength{\itemsep}{0cm}
\setlength{\leftmargin}{0.7cm}
\setlength{\labelsep}{0.2cm}}}
{\end{list}}

\newenvironment{proof-theo}{\noindent{\it Proof of Theorem \ref{strength-Plachky-Steinebach}.}}{\fbox{}\\}

\newenvironment{proof-corollary}{\noindent{\it Proof of Corollary \ref{strength-Plachky-Steinebach-case-a)}.}}{\fbox{}\\}

\newenvironment{proof-claim 1}{\noindent{\it Proof of Claim \ref{claim 1}.}}{\fbox{}\\}

\newenvironment{proof-claim 2}{\noindent{\it Proof of Claim  \ref{claim 2}.}}{\fbox{}\\}

\newenvironment{proof-claim 3}{\noindent{\it Proof of Claim  \ref{claim 3}.}}{\fbox{}\\}

\newtheorem{theorem}{Theorem}

\newtheorem{corollary}{Corollary}
\newtheorem{lemma}{Lemma}
\newtheorem{claim}{Claim}

\theoremstyle{remark}
\newtheorem{remark}{Remark}

\theoremstyle{definition}

\theoremstyle{definition}

\def\N{{\mathbb N}}

\def\R{{\mathbb R}}

\newcommand{\ie}{{\it i.e.\/}\ }
\newcommand{\eg}{{\it e.g.\/}\ }
\newcommand{\cf}{{\it cf.\/}\ }

\def\text{\hbox}

\title[]{Improvements  of Plachky-Steinebach theorem}

\begin{document}

\author{Henri Comman$^\dag$}
\address{Pontificia Universidad Cat\'{o}lica de Valparaiso,  Avenida Brasil 2950, Valparaiso, Chile}
\email{henri.comman@pucv.cl}
\thanks{$\dag$ Partially supported by FONDECYT grant 1120493.}


\begin{abstract}
We show that the conclusion of Plachky-Steinebach  theorem  holds true 
for intervals of the form $\left]\overline{L}_r'(\lambda),y\right[$,  where     $\overline{L}_r'(\lambda)$  is   the right derivative  (but not necessarily  a derivative)  of the generalized log-moment generating function $\overline{L}$  at some $\lambda> 0$ and  $y\in\ \left]\overline{L}_r'(\lambda),+\infty\right]$, 
under the only two following  conditions:  $(a)$  $\overline{L}'_r(\lambda)$ is  a limit point of  the set  $\left\{\overline{L}'_r(t):t>\lambda\right\}$,    $(b)$ $\overline{L}(t_i)$ is a limit  with   $t_i$ belonging to  some  decreasing sequence  converging to $\sup\{t>\lambda: \overline{L}_{\mid ]\lambda,t]}\ \textnormal{is affine}\}$. 
   By replacing $\overline{L}_r'(\lambda)$   by $\overline{L}_r'(\lambda^+)$, 
     the above result  extends verbatim  to the case $\lambda=0$ (replacing $(a)$ by the right continuity of $\overline{L}$ at zero when 
      $\overline{L}_r'(0^+)=-\infty$).    No  hypothesis  is made on $\overline{L}_{]-\infty,\lambda[}$  (\eg     $\overline{L}_{]-\infty,\lambda[}$ may be  the constant $+\infty$ when $\lambda=0$);  $\lambda\ge 0$ may  be a non-differentiability point of $\overline{L}$ and moreover a limit point of non-differentiability points of $\overline{L}$; $\lambda=0$ may be  a left and right  discontinuity point of $\overline{L}$. The map $\overline{L}_{\mid ]\lambda,\lambda+\varepsilon[}$ may fail to be strictly convex for all $\varepsilon>0$.   If we drop the  assumption $(b)$,   then the same  conclusion holds  with upper limits in place of limits.
    Furthermore, the foregoing  is valid  for general nets $(\mu_\alpha,c_\alpha)$ of Borel probability measures and powers (in place of the sequence $(\mu_n,n^{-1})$) and replacing the intervals $\left]\overline{L}_r'(\lambda^+),y\right[$  by $]x_\alpha,y_\alpha[$ or $[x_\alpha,y_\alpha]$,  where  $(x_\alpha,y_\alpha)$  is   any  net  such that $(x_\alpha)$
   converges  to  $\overline{L}_r'(\lambda^+)$ and  $\liminf_\alpha y_\alpha>\overline{L}_r'(\lambda^+)$.
    \end{abstract}



\maketitle

\section{Introduction and statement of the results}\label{intro}

 Let $(\mu_\alpha,c_\alpha)$ be a net where $\mu_\alpha$ is a  
 Borel probability measure on $\R$, $c_\alpha>0$ and $(c_\alpha)$ converges to zero. Let   $\overline{L}$ 
 be the  generalized  log-moment generating function associated with $(\mu_\alpha,c_\alpha)$, defined 
  for each  $t\in\R$ by
 \begin{equation}\label{introduction-eq30}
  \overline{L}(t)=\limsup_\alpha c_\alpha\log\int_\R e^{c_\alpha^{-1} t x}\mu_\alpha(dx);
   \end{equation}
    when the above upper limit is a limit we  write $L(t)$ in place of $\overline{L}(t)$. Let $\lambda>0$ and assume that the derivative map    $L'$   exists   and is  strictly monotone  on a  neighbourhood of $\lambda$ (in particular,  $L(t)$ exists in $\R$  for all $t$ in this neighbourhood).  When  $(\mu_\alpha,c_\alpha)$ is a sequence of the form $(\mu_n,n^{-1})$, Plachky-Steinebach theorem (\cite{Plachky_Steinebach(75)PerMatHung6}) asserts that  for each sequence $(x_n)$ of real numbers converging to $L'(\lambda)$ we have 
   \begin{equation}\label{introduction-eq40}
   \lim_n n^{-1}\log\mu_n(]x_n,+\infty[)=L(\lambda)-\lambda L'(\lambda). 
   \end{equation}
     (Note that  $L(\lambda)-\lambda L'(\lambda)=-\overline{L}^*(L'(\lambda))=\inf_{t> 0}\{\overline{L}(t)-t L'(\lambda)\}$,  where $\overline{L}^*$ denotes the Legendre-Fenchel transform of $\overline{L}$.)

       In this paper,   we improve   the above result in various ways:         First, we weaken all the hypotheses by 
       (a)  replacing the differentiability  of $\overline{L}$ on a neighbourhood of $\lambda$ by  the condition  that   $\overline{L}'_r(\lambda)$ is  a limit point of  the set  $\left\{\overline{L}'_r(t):t>\lambda\right\}$, where $\overline{L}'_r$ denotes the right derivative map of $\overline{L}$, (b)
       removing entirely the strict convexity hypothesis, (c) allowing $\lambda=0$ (with the only hypothesis of right continuity of $\overline{L}$ at zero when $\overline{L}'_r(0^+)=-\infty$), 
        (d) requiring the existence of $L$ only on  a suitable sequence, (e) allowing general nets $(\mu_\alpha,c_\alpha,x_\alpha)$  in place of the sequence $(\mu_n,n^{-1},x_n)$;     second, we strengthen the conclusion by allowing a more general class of intervals in the left hand side of (\ref{introduction-eq40}); third, we  generalize all the above by giving  a version with upper limits.

     For each $\lambda\ge 0$,      our result covers all situations except  when the  condition  mentioned in $(a)$ or in $(c)$ fails, \ie  excluding    the  following ones (\cf Lemma \ref{existence L on t}): 
      \begin{itemize}
\item[$(i)$] The map ${\overline{L}'_r}_{\mid ]\lambda,+\infty[}$ takes only infinite values (equivalently,  either $\overline{L}_{\mid [0,+\infty[}$ is improper, or $\overline{L}(t)=+\infty$ for all $t>\lambda$).
\item[$(ii)$]  There exists $T\in\ ]\lambda,+\infty]$ such that  $\overline{L}_{\mid ]\lambda,T[}$ is affine,  with  $\overline{L}'_r(\lambda^+)<\overline{L}'_r(T)<+\infty$ when $T<+\infty$.
\item[$(iii)$] $\lambda=0$, $\overline{L}_r'(0^+)=-\infty$ and $\overline{L}$ is right discontinuous at zero.
\end{itemize}
(In particular, in the most common case where $\overline{L}_{[\lambda,\lambda+\varepsilon[}$ is real-valued and continuous for some $\varepsilon>0$,  only $(ii)$ is excluded.)
 Equivalently,  Theorem \ref{strength-Plachky-Steinebach} applies in  (and only in) all the  cases 
 where   $\overline{L}'_r(\lambda^+)=\lim_i \overline{L}'_r(\lambda_i)$  for some   sequence  $(\lambda_i)$   fulfilling  eventually  
 \begin{equation}\label{introduction-eq-60}
 \overline{L}'_r(\lambda^+)<\overline{L}'_r(\lambda_i)<+\infty,
 \end{equation}
with $\overline{L}(0^+)=0$ when $\lambda=0$ and $\overline{L}_r'(0^+)=-\infty$ (note that the above condition  is satisfied under the hypotheses of Plachky-Steinebach theorem);  then,  there is
 a unique  real number ${\tilde{\lambda}}\ge\lambda$  fulfillling $\overline{L}'_r(\lambda^+)=\overline{L}'_r({\tilde{\lambda}}^+)$ and $\overline{L}'_r(t)>\overline{L}'_r(\lambda^+)$ for all $t>{\tilde{\lambda}}$; in fact, 
${\tilde{\lambda}}=\sup\{t>\lambda: \overline{L}_{\mid ]\lambda,t]}\ \textnormal{is affine}\}=\lim_i \lambda_i$ for every sequence as above (\cf Lemma \ref{existence L on t}).

  Let us  emphasize  that    no hypothesis  is made on the map $\overline{L}_{\mid ]-\infty,\lambda[}$. 
 The proof deals  only with  the properties of the right derivative map  ${\overline{L}'_r}_{\mid [\lambda,\lambda+\varepsilon[}$;  
in particular, the following special situations may arise: 
  
 \begin{itemize} 
  \item  $\lambda$ may be a non-differentiability point of $\overline{L}$,  and moreover a limit point of non-differentiability points of $\overline{L}$
  (\eg when  $\lambda={\tilde{\lambda}}$ and eventually  $\lambda_i$ is a non-differentiability point of $\overline{L}$).
  \item   $\overline{L}$ may be  left  discontinuous  at $\lambda$ when  $\lambda=0$  (which is the case when $\overline{L}_r'(0)=-\infty$) and  moreover  right discontinuous  at $0$    (when $\overline{L}_r'(0)=-\infty<\overline{L}'_r(0^+)$); \cf Remark \ref{remark-difference lambda-positive-vs-lambda-zero}.
  \item    $\overline{L}_{\mid ]\lambda,\lambda+\varepsilon[}$  may  not be  strictly convex for all $\varepsilon>0$ (\eg when  $\lambda={\tilde{\lambda}}$ and  eventually $\overline{L}_{\mid [\lambda_{i+1},\lambda_i]}$ is affine);  when $\lambda<{\tilde{\lambda}}$, $\overline{L}$ is not strictly convex in any neighbourhood of $\lambda$ or ${\tilde{\lambda}}$.
  \item As regards  the strong  version with limits, for each  $t<\lambda$,  $\overline{L}(t)$ may not   be a limit. 
 \end{itemize}
  The foregoing contrasts  sharply with the hypotheses of Plachky-Steinebach theorem, which  require $\lambda>0$ and  some $\varepsilon>0$ such that     $L(t)$  exists in $\R$ for all $t\in\  ]\lambda-\varepsilon,\lambda+\varepsilon[$ and   $L_{\mid  ]\lambda-\varepsilon,\lambda+\varepsilon[}$  is strictly convex and differentiable.
  
   Theorem \ref{strength-Plachky-Steinebach} \footnote{Some words of caution about the notations: First, 
 the net  $(\mu_\alpha)$ is  considered as a net of Borel  measures on $[-\infty,+\infty]$  in order to give a sense to $\mu_\alpha([x_\alpha,y_\alpha])$ when $x_\alpha=-\infty$ or $y_\alpha=+\infty$; second, when   $\overline{L}_r'(\lambda^+)=-\infty$, we adopt   the convention $0\cdot\overline{L}'_r(\lambda^+)=0$; third,  in order to simplify the wording of the theorem, 
     we do not distinguish with  the notations  between the cases $\lambda>0$ and $\lambda=0$, so that we use  $\overline{L}_r'(\lambda^+)$ and $\overline{L}(\lambda^+)$,  although  $\overline{L}_r'(\lambda^+)=\overline{L}_r'(\lambda)$ and $\overline{L}(\lambda^+)=\overline{L}(\lambda)$ when $\lambda>0$, and  $\overline{L}_r'(\lambda^+)=-\infty$ implies $\lambda=0$ (\cf Lemma \ref{improper-L}).}
     below constitutes the first general result giving the upper limit (and a fortiori,  the  limit) of 
   $(c_\alpha\log\mu_\alpha([x_\alpha,y_\alpha]))$ or   $(c_\alpha\log\mu_\alpha(]x_\alpha,y_\alpha[)$ in terms of $\overline{L}^*$, when $\lim_\alpha x_\alpha$  is not a value of the derivative map  of the  generalized  log-moment generating function (\cf Remark \ref{remark G-E theo}).

   \begin{theorem}\label{strength-Plachky-Steinebach}
Let $\lambda\ge 0$.  When $\overline{L}_r'(\lambda^+)>-\infty$ we assume that  $\overline{L}'_r(\lambda^+)$ is  a limit point of  the set  $\left\{\overline{L}'_r(t):t>\lambda\right\}$;  when $\overline{L}_r'(\lambda^+)=-\infty$ we assume that $\overline{L}(\lambda)=\overline{L}(\lambda^+)$.  For each net  $(x_\alpha,y_\alpha)$ in $[-\infty,+\infty]^2$  such that  $(x_\alpha)$ converges   to  $\overline{L}_r'(\lambda^+)$  and       $\liminf_\alpha y_\alpha>\overline{L}_r'(\lambda^+)$ we have 
        \begin{multline*}
      \limsup_\alpha c_\alpha\log\mu_\alpha([x_\alpha,y_\alpha])
=   \limsup_\alpha c_\alpha\log\mu_\alpha(]x_\alpha,y_\alpha[) =  \overline{L}(\lambda^+)-\lambda\overline{L}_r'(\lambda^+)
  \\
  = \left\{\begin{array}{ll}
 -\overline{L}^*(\overline{L}_r'(\lambda^+))= \inf_{t>0}\{\overline{L}(t)-t\overline{L}_r'(\lambda^+)\}  & if\ \overline{L}_r'(\lambda^+)\neq-\infty,
  \\ 
-\lim_\alpha \overline{L}^*(x_\alpha)= 0 & if\ \overline{L}_r'(\lambda^+)=-\infty.
      \end{array}
    \right.
    \end{multline*}  
       Furthermore, if $L(t_i)$ exists    for a   sequence $(t_i)$ in $]{\tilde{\lambda}},+\infty[$  converging  to ${\tilde{\lambda}}$, then the  above upper limits are  limits (with  ${\tilde{\lambda}}=\sup\{t>\lambda: \overline{L}_{\mid ]\lambda,t]}\ \textit{is affine}\}$).  
      \end{theorem}

         It  is possible   that for some $\varepsilon>0$,  $\lambda$ is the only point in $[0,\lambda+\varepsilon]$ 
      where          
         Theorem  \ref{strength-Plachky-Steinebach} applies
                 (\eg when $\lambda={\tilde{\lambda}}$,    $\overline{L}$ is not differentiable at $\lambda$,  and  eventually  $\lambda_i$ appearing in (\ref{introduction-eq-60}) is a non-differentiability point of $\overline{L}$ and 
   $\overline{L}_{[\lambda_{i+1},\lambda_i]}$ is affine; although this may seem to be an extreme case,   we show in Appendix A that there are plenty of such examples); in the light of the above, 
        Theorem  \ref{strength-Plachky-Steinebach} may be thought of as a pointwise version of  Plachky-Steinebach theorem.

 Hereafter, we focus  on the opposite situation, motivated by the following observation: 
   If  $L$ exists  as a    differentiable and strictly convex map  on some open  interval containing $\lambda>0$ and  zero (says $]-\varepsilon,\lambda+\varepsilon[$), then the hypotheses of Plachky-Steinebach theorem
           hold for all $t\in\ ]0,\lambda+\varepsilon[$ (in place of $\lambda$),  and  thus it  applies to  every sequence  converging to $L'(t)$ for all $t\in\ ]0,\lambda+\varepsilon[$; since 
        ${\overline{L}'}_{\mid ]0,\lambda+\varepsilon[}$ is in this case an increasing  homeomorphism,  
          the limit obtained, as a function of $L'(t)$,  is exactly $-{L^*}_{\mid ]L_r'(0),L_l'(\lambda+\varepsilon)[}$ (where $\overline{L}'_l$ denotes the left derivative map of  $\overline{L}$); 
          a    weak version  of this (with  the  constant sequence $(x_n)\equiv L'(t)$) has been 
      widely used in the context  of  dynamical systems 
                (\eg \cite{Hennion_Herve}, Lemma 13.2;  \cite{Bougerol_Lacroix}, Lemma 6.2; \cite{Chazottes_Ugalde}, Corollary 4.3; \cite{Chazottes_Leplaideur},  Theorem 2.2; \cite{Aimino_Vaienti}, Proposition 7).

              A much stronger result can be derived from Theorem \ref{strength-Plachky-Steinebach}, as shows the 
             following Corollary \ref{strength-Plachky-Steinebach-case-a)}; 
    aside the extension with upper limits and the possibility to consider nets in place of sequences,       the  improvements  are obtained first, by  allowing a more   general class of  intervals  in the conclusion, and second, by  weakening the hypotheses as follows:
    \begin{itemize}
     \item  The differentiability  of $\overline{L}_{\mid ]-\varepsilon,\lambda+\varepsilon[}$ is replaced by  the differentiability  of $\overline{L}_{\mid]\lambda,\lambda+\varepsilon[}$.
    \item   The strict  convexity of $\overline{L}_{\mid ]-\varepsilon,\lambda+\varepsilon[}$ is replaced by the condition that 
    ${\overline{L}'}_{\mid ]\lambda,\lambda+\varepsilon[}$ does not attain its supremum (equivalently, for each $t\in\ ]\lambda,\lambda+\varepsilon[$ the map $\overline{L}_{\mid[t,\lambda+\varepsilon[}$ is not affine). 
        
     \medskip
 
    (Note first, that this condition  is obviously fulfilled when 
    $\overline{L}_{\mid ]t,\lambda+\varepsilon[}$ is strictly convex for some $t\in[\lambda,\lambda+\varepsilon[$, and second, 
    it  is far weaker than strict  convexity: There may be an infinite  countable set   $\{S_i:i\in\N\}$ of non-empty mutually disjoint intervals    included in $]\lambda,\lambda+\varepsilon[$ on which $\overline{L}$  is affine;  when $\sup_i \sup S_i=\lambda+\varepsilon$, 
     the above condition implies     $\sup S_i<\lambda+\varepsilon$ for all $i\in\N$.) 
     
       \medskip

              \item The possibility to take $\lambda=0$, including the cases $\overline{L}'(0^+)=-\infty$  and  $\overline{L}'(0)=-\infty<\overline{L}'(0^+)$.
              
    \item  The version with limits only requires the existence  of $L$ on a sequence $(t_i)$ in $]\tilde{t},+\infty[$  converging  to $\tilde{t}$, for all $t\in[\lambda,\lambda+\varepsilon[$; note that the set $[\lambda,\lambda+\varepsilon[\setminus\{\tilde{t}:t\in[\lambda,\lambda+\varepsilon[\}$ contains all the  open intervals  on which $\overline{L}$ is affine. 
    
             \end{itemize}

                  \begin{corollary}\label{strength-Plachky-Steinebach-case-a)}
 Let $\lambda\ge 0$ and let $\varepsilon>0$. We assume that  $\overline{L}_{\mid ]\lambda,\lambda+\varepsilon[}$ is       differentiable and $\sup\left\{\overline{L}'(t): t\in\ ]\lambda,\lambda+\varepsilon[\right\}$ is not a maximum; when $\lambda=0$ and $\overline{L}_r'(0^+)=-\infty$,  we further assume that $\overline{L}$  is right continuous at zero. 
          \begin{nitemize}
          \item[a)]  For each $z\in\left[\overline{L}_r'(\lambda^+),\overline{L}_l'(\lambda+\varepsilon)\right[$,            
          there exists $t_z\in[\lambda,\lambda+\varepsilon[$ fulfilling $\overline{L}_r'(t_z^+)=z$, and for each  such  $t_z$, for each net   $(x_{\alpha,z})$ in $[-\infty,+\infty]$ converging to $z$, and for each net 
          $(y_{\alpha,z})$ in $[-\infty,+\infty]$ fulfilling  $\liminf_\alpha y_{\alpha,z}>z$, we have
      \begin{multline*}
      \limsup_\alpha c_\alpha\log\mu_\alpha([x_{\alpha,z},y_{\alpha,z}])
=\limsup_\alpha c_\alpha\log\mu_\alpha(]x_{\alpha,z},y_{\alpha,z}[)
       =   \overline{L}(t_z^+)-t_z z
       \\
       =\left\{\begin{array}{ll}
   -\overline{L}^*(z)=\inf_{t>0}\{\overline{L}(t)-tz\}  & if\ z\neq -\infty
  \\ 
-\lim_\alpha \overline{L}^*(x_{\alpha,z})=0 & if\ z=-\infty\ \  (\ie \lambda=0\ and\  z=\overline{L}_r'(0^+)=-\infty).
      \end{array}
    \right.
\end{multline*}
 Furthermore, if $L(t_{z,i})$ exists    for a   sequence $(t_{z,i})$ in $]\tilde{t_z},+\infty[$  converging  to $\tilde{t_z}$, then the  above upper limits are  limits (with  $\tilde{t_z}=\sup\{t>t_z: \overline{L}_{\mid ]t_z,t]}\ \textit{is affine}\}$).

     \item[b)]  The map $\left[\overline{L}_r'(\lambda^+),\overline{L}_l'(\lambda+\varepsilon)\right[\ni z\rightarrow\overline{L}(t_z^+)-t_z z$  is 
              $]-\infty,0]$-valued, 
             strictly decreasing, continuous, unbounded below if and only if  $\overline{L}_l'(\lambda+\varepsilon)=+\infty$,  and  vanishes at $\overline{L}_r'(\lambda^+)$  if and only if either $\lambda>0$ and $\overline{L}$ is differentiable at $\lambda$ and linear on $[0,\lambda]$, or 
$\lambda=0$ and $\overline{L}$  is right continuous at zero; its  restriction to 
$\left]\overline{L}_r'(\lambda^+),\overline{L}_l'(\lambda+\varepsilon)\right[$ is    strictly concave, and it is 
 furthermore differentiable when
                 $\overline{L}_{\mid ]{\tilde{\lambda}},\lambda+\varepsilon[}$ is   strictly convex (with  ${\tilde{\lambda}}=\sup\{t>\lambda: \overline{L}_{\mid ]\lambda,t]}\ \textit{is affine}\}$). 
\end{nitemize}
              \end{corollary}

             \begin{remark}\label{general hypothesis  theo}
        We have     $\overline{L}_r'(\lambda^+)=-\infty$ only if $\lambda=0$, in which case 
                    the  hypothesis of right continuity of $\overline{L}$ at zero in  Theorem \ref{strength-Plachky-Steinebach} (and Corollary \ref{strength-Plachky-Steinebach-case-a)}) cannot be removed; indeed, suppose that $\overline{L}(0^+)\neq 0$, and let $c\in\ ]0,1[$. 
    Since for each index $\alpha$ the  measure $\mu_\alpha$ is tight,   there exists  a net $(x_\alpha)$ of real numbers converging to $-\infty$ such that  $\mu_\alpha([x_\alpha,-x_\alpha])>c$, hence  
     \[\lim_\alpha c_\alpha\log\mu_\alpha([x_\alpha,+\infty[)=0\neq\overline{L}(0^+),\]
 and the conclusion  does not hold with $(x_\alpha)$. The hypothesis for the case  $\overline{L}_r'(\lambda^+)>-\infty$ also holds  when   $\overline{L}_r'(\lambda^+)=-\infty$ (\cf Lemma \ref{existence L on t}$b)$).
  \end{remark}

\begin{remark}\label{vanishing-limit}
The upper limits in Theorem \ref{strength-Plachky-Steinebach} vanish if and only if one of the following cases occurs:
\begin{itemize}
\item $\lambda>0$, $\overline{L}$ is differentiable at $\lambda$  and  linear on $[0,\lambda]$;
\item $\lambda=0$ and $\overline{L}$  is right continuous at zero.
\end{itemize}
\end{remark}

\begin{remark}\label{validity of lim_alpha L*(x_alpha)}
Under the hypotheses of Theorem \ref{strength-Plachky-Steinebach},
 the equality  $\overline{L}^*(\overline{L}_r'(\lambda^+))=\lim_\alpha \overline{L}^*(x_\alpha)$   fails when eventually 
$x_\alpha$  does not belong to  the effective domain of $\overline{L}^*$, which happens if and only if $L_{\mid ]-\infty,\lambda[}$ is linear  with slope $\overline{L}_r'(\lambda^+)$ (equivalently,  $\overline{L}_r'(\lambda^+)$ is the left end-point of the effective domain of $\overline{L}^*$)
 and eventually $x_\alpha<\overline{L}_r'(\lambda^+)$. 
\end{remark}

      \begin{remark}\label{remark-difference lambda-positive-vs-lambda-zero}
      Regarding the proof as well as the wording of Theorem \ref{strength-Plachky-Steinebach},  the main difference between the cases $\lambda>0$ and $\lambda=0$  stems from the fact that  when  the map $\overline{L}_{\mid [0,+\infty[}$ is proper (which is implied by the hypotheses, \cf Lemma \ref{existence L on t}),  $0$ is the only  nonnegative real number $\lambda$ in the effective domain of $\overline{L}_{\mid [0,+\infty[}$ for which: 
  \begin{itemize}
  \item   $\overline{L}'_r(\lambda)$ or  $\overline{L}'_r(\lambda^+)$  may take the value $-\infty$;
    \item  $\overline{L}'_r(\lambda)$ may differ from  $\overline{L}'_r(\lambda^+)$;
    \item $\overline{L}'_r(\lambda^+)<+\infty$ and $\overline{L}$ may be right discontinuous at $\lambda$. 
      \end{itemize} (\cf Lemma \ref{improper-L} and Lemma  \ref{right-continuity of right derivative}; for instance,   the conclusion  in the  last assertion of Lemma \ref{improper-L} is not  true  for $\lambda=0$). Theorem \ref{strength-Plachky-Steinebach} and Corollary \ref{strength-Plachky-Steinebach-case-a)}
 include the  case $\lambda=0$ and  $\overline{L}'_r(0)=-\infty<\overline{L}'_r(0^+)$, which implies the  left  and right discontinuities   of  $\overline{L}$ at zero (\cf Lemma \ref{right-continuity of right derivative}); the case   $\lambda=0$ and     $\overline{L}'_r(0^+)=-\infty$  implies the left discontinuity  of  $\overline{L}$ at zero (in both cases, we have  $\overline{L}_{]-\infty,0[}\equiv+\infty$).
\end{remark}

\begin{remark}\label{remark G-E theo}
The standard version of   G\"{a}rtner-Ellis theorem   is unworkable    when $\overline{L}$ is not differentiable at $\lambda$: Indeed, 
  the  main hypothesis (\ie essentially smoothness)  implies the differentiability of $\overline{L}$ on the interior of its effective domain
   (beside requiring    the  existence of $L$ on an open interval containing  $0$); 
    the same  applies to the  variant of   G\"{a}rtner-Ellis theorem given by Theorem 5.1 of \cite{O'Brien(1995)SPA57} (which allows  $L$ to exists only on some open  interval not necessarily containing $0$) since it   also requires  the essential smoothness. 
    Corollary 1 of  \cite{Comman(2009)CM1}  strengthens both above versions, but,  although weaker than essential smoothness, the general  hypothesis   
   is still  a global condition requiring the  existence of $L$ on some open interval and 
relating the range of the one-sided  derivatives of $\overline{L}$  with the effective domain of $\overline{L}^*$,  and thus it is of no use here.
  \end{remark}

 \section{Proofs}

  Recall that $\overline{L}$ is a  $[-\infty,+\infty]$-valued convex function (\cite{Dembo-Zeitouni}, \cite{Comman(2007)STAPRO77}),  and such a function is said to be proper when it is $]-\infty,+\infty]$-valued and takes at least one finite value; note that 
   $\overline{L}(0)=0$.  We denote by $\overline{L}^*$ and $\overline{L}_r'$ (resp. $\overline{L}_l'$) respectively   the Legendre-Fenchel transform and right (resp. left) derivative map  of $\overline{L}$, where for each $t\ge 0$ we put  $\overline{L}_r'(t)=+\infty$ (resp. $\overline{L}_r'(t)=-\infty$) when  $\overline{L}(s)=+\infty$  for all $s>t$ (resp. $\overline{L}(s)=-\infty$ for some   $s>t$); note that $\overline{L}^*$ is $[0,+\infty]$-valued since $\overline{L}(0)=0$.
   The basic properties of the map ${\overline{L}'_r}_{\mid ]0,+\infty[}$ are summarized in  Lemmas \ref{existence L on t}$a)$,  Lemma    \ref{improper-L} and  Lemma \ref{right-continuity of right derivative}; 
    in particular,     ${\overline{L}'_r}_{\mid ]0,+\infty[}$  is  non-decreasing so that  the quantity  $\overline{L}'_r(\lambda^+)=\lim_{t\rightarrow\lambda,t>\lambda} \overline{L}'_r(t)$ is well-defined for all $\lambda\ge 0$.
 
 Let $\lambda\ge 0$.   All the lemmas below hold for any net $(\mu_\alpha,c_\alpha)$ as in (\ref{introduction-eq30}); the hypotheses of Theorem \ref{strength-Plachky-Steinebach} are  made only   in the last part of the proof.

\begin{lemma}\label{existence L on t}
\begin{nitemize}
\item[$a)$]The map  $\overline{L}_{\mid [0,+\infty[}$ is improper if and only if  there exists $T\in\ ]0,+\infty]$  such   that  $\overline{L}_{\mid ]0,T[}={\overline{L}'_r}_{\mid [0,T[}\equiv-\infty$ and $\overline{L}_{\mid ]T,+\infty[}={\overline{L}'_r}_{\mid [T,+\infty[}\equiv+\infty$. 
\item[$b)$] If  $\overline{L}_{\mid [0,+\infty[}$ is proper, then 
one and only one   of the following cases holds:
\begin{itemize}
\item[$(i)$] There exists a sequence $\left(\overline{L}'_r(\lambda_i)\right)$ in $\left]\overline{L}'_r(\lambda^+),+\infty\right[$ converging to 
$\overline{L}'_r(\lambda^+)$.
\item[$(ii)$] There exists $\varepsilon>0$ such that  $\overline{L}_{\mid ]\lambda,\lambda+\varepsilon]}$ is affine and $\overline{L}'_r(\lambda^+)<\overline{L}'_r(\lambda+\varepsilon)<+\infty$. 
\item[$(iii)$] $\overline{L}_{\mid ]\lambda,+\infty[}$ is affine.
\item[$(iv)$] $\overline{L}(\lambda)\in\R$ and $\overline{L}(t)=+\infty$ for all $t>\lambda$;
\item[$(v)$] $\overline{L}(t)=+\infty$ for all $t\ge\lambda$.
\end{itemize}
In particular,  $(i)$ or  $(ii)$ or   $(iii)$ holds if and only if  there exists $\varepsilon>0$ such that $\overline{L}_{\mid [\lambda,\lambda+\varepsilon]}$ is real-valued, bounded,   and  $\overline{L}_{\mid ]\lambda,\lambda+\varepsilon]}$  (resp. $\overline{L}_{\mid [\lambda,\lambda+\varepsilon]}$  when $\lambda>0$) continuous.
Furthermore,  $(i)$ holds if and only if there exists ${\tilde{\lambda}}\in [\lambda,+\infty[$ fulfilling 
  $\overline{L}'_r(\lambda^+)=\overline{L}'_r({\tilde{\lambda}}^+)$ and $\overline{L}'_r(t)>\overline{L}'_r(\lambda^+)$ for all $t>{\tilde{\lambda}}$; such a ${\tilde{\lambda}}$ is unique and given by
  \[{\tilde{\lambda}}=\sup\{t>\lambda: \overline{L}_{\mid ]\lambda,t]}\ \textit{is affine}\}=\lim_i\lambda_i\] for every  sequence $(\lambda_i)$ as in $(i)$.  
 \end{nitemize}
\end{lemma}

\begin{proof}
$a)$ It is a direct consequence of the definitions together with the convexity of $\overline{L}$ and  the fact that $\overline{L}(0)=0$. 

$b)$ Assume that $\overline{L}_{\mid [0,+\infty[}$ is proper. 
If $\overline{L}(\lambda)=+\infty$, then    $\overline{L}(t)=+\infty$ for all $t\ge\lambda$ (because $\overline{L}$ is   convex and $\overline{L}(0)=0$) and $(v)$ holds.
Assume  now that  $\overline{L}(\lambda)<+\infty$. Then 
 $\overline{L}(\lambda)\in\R$ because   $\overline{L}_{\mid [0,+\infty[}$ is proper.
Since $\overline{L}$ is convex and $\overline{L}(0)=0$,  we have  either $\overline{L}(t)=+\infty$ for all $t>\lambda$ and $(iv)$ holds, or  
 $\overline{L}_{\mid [\lambda,\lambda+\delta[}$ is  real-valued for some $\delta>0$,  in which  case 
 $\overline{L}'_r(t)\in\R$ for all $t\in\ ]\lambda,\lambda+\delta[$.  
Assume that $(i)$ does not hold, \ie 
 $\overline{L}'_r(\lambda^+)<\inf\left\{\overline{L}'_r(t)\in\ \left]\overline{L}'_r(\lambda^+),+\infty\right[:t>\lambda\right\}$.
  If $\left\{\overline{L}'_r(t)\in\ \left]\overline{L}'_r(\lambda^+),+\infty\right[:t>\lambda\right\}=\emptyset$, then   $\overline{L}'_r(t)=\overline{L}'_r(\lambda^+)$  for all $t>\lambda$,  and $(iii)$ holds.  If $\left\{\overline{L}'_r(t)\in\ \left]\overline{L}'_r(\lambda^+),+\infty\right[:t>\lambda\right\}\neq\emptyset$, then  
  there exists $\varepsilon\in\ ]0,\delta[$ such that $\overline{L}'_r(\lambda^+)<\overline{L}'_r(\lambda+\varepsilon)$ and 
 $\overline{L}'_r(\lambda^+)=\overline{L}'_r(t)$ for all $t\in\ ]\lambda,\lambda+\varepsilon[$, so that   $(ii)$ holds. The first assertion is proved.

 Assume that either $(i)$,  $(ii)$, or $(iii)$ holds.  There exists $\varepsilon>0$ such that
 ${\overline{L}'_r}_{\mid ]\lambda,\lambda+\varepsilon[}$ is real-valued,  hence $\overline{L}_{\mid [\lambda,\lambda+\varepsilon[}$  is real-valued. 
 The map 
   $\overline{L}_{\mid ]\lambda,\lambda+\varepsilon/2]}$   (resp. $\overline{L}_{\mid [\lambda,\lambda+\varepsilon/2]}$ when $\lambda>0$) is continuous since  $]\lambda,\lambda+\varepsilon/2]$  (resp. $[\lambda,\lambda+\varepsilon/2]$ when $\lambda>0$)
 belongs to the interior of the effective domain of $\overline{L}_{\mid[0,+\infty[}$. 
 The boundness of
 $\overline{L}_{\mid [\lambda,\lambda+\varepsilon/2]}$  follows from the continuity  when $\lambda>0$, and the  boundness of
 $\overline{L}_{\mid ]\lambda,\lambda+\varepsilon/2]}$  when $\lambda=0$ follows 
 from the convexity and the fact that $\overline{L}(0)=0$. 
 Conversely, if  $\overline{L}_{\mid [\lambda,\lambda+\varepsilon[}$ is real-valued for some $\varepsilon>0$, then the first assertion 
 implies that 
 either $(i)$,  $(ii)$, or $(iii)$ holds. 
  The second assertion is proved. 
 
 Assume that there exists 
    $s\in[\lambda,+\infty[$ fufilling $\overline{L}'_r(\lambda^+)=\overline{L}'_r(s^+)$ and $\overline{L}'_r(t)>\overline{L}'_r(\lambda^+)$ for all $t>s$. 
     We have   $\overline{L}'_r(\lambda^+)=\lim_{t\rightarrow s,t>s}\overline{L}'_r(t)$ 
    hence $(i)$ holds.

 Assume that $(i)$ holds.  Put ${\tilde{\lambda}}=\sup\{t>\lambda: \overline{L}_{\mid ]\lambda,t]}\ \textnormal{is affine}\}$.
 If  ${\tilde{\lambda}}=+\infty$, then $\overline{L}'_r(t)=\overline{L}'_r(\lambda^+)$ for all $t>\lambda$, which contradicts $(i)$, hence  ${\tilde{\lambda}}\in[\lambda,+\infty[$.  
  Let $t>{\tilde{\lambda}}$. Since $\overline{L}_{\mid ]\lambda,t]}$ is not affine, we have $\overline{L}'_r(t)>\overline{L}'_r(\lambda^+)$.
  Suppose that $\overline{L}'_r(\lambda^+)<\overline{L}'_r({\tilde{\lambda}}^+)$; by   $(i)$ there exists $\overline{L}'_r(\lambda_i)$ such that \[\overline{L}'_r(\lambda^+)<\overline{L}'_r(\lambda_i)<\overline{L}'_r({\tilde{\lambda}}^+),\] which implies $\lambda<\lambda_i<{\tilde{\lambda}}$  and 
   contradicts the fact that $\overline{L}_{]\lambda,{\tilde{\lambda}}[}$ is affine; therefore, $\overline{L}'_r(\lambda^+)=\overline{L}'_r({\tilde{\lambda}}^+)$.

Let $(\lambda_i)$ be a sequence as in $(i)$. For each $\varepsilon>0$  we have eventually \[\overline{L}'_r(\lambda^+)=\overline{L}'_r({\tilde{\lambda}}^+)<\overline{L}'_r(\lambda_i)<\overline{L}'_r({\tilde{\lambda}}+\varepsilon)\]
hence eventually 
${\tilde{\lambda}}<\lambda_i<{\tilde{\lambda}}+\varepsilon$, which implies
 ${\tilde{\lambda}}=\lim_i \lambda_i$. 
 
  The equalities  $\overline{L}'_r(\lambda^+)=\overline{L}'_r({\tilde{\lambda}}^+)=\overline{L}'_r(s^+)$ implies $\overline{L}_{\mid]\lambda,s]}$ affine, hence $s\le {\tilde{\lambda}}$ by definition of ${\tilde{\lambda}}$. Since $\overline{L}'_r(t)>\overline{L}'_r(\lambda^+)=\overline{L}'_r({\tilde{\lambda}}^+)$ for all $t>s$, we have $\overline{L}'_r(s^+)\ge\overline{L}'_r({\tilde{\lambda}}^+)$ hence $s\ge {\tilde{\lambda}}$; 
therefore, $s={\tilde{\lambda}}$. The proof of the third assertion is complete. 
  \end{proof}

\begin{lemma}\label{improper-L}
The following statements are equivalent:
\begin{itemize}
\item[$(i)$]   $\overline{L}_{\mid [0,+\infty[}$ is proper;
\item[$(ii)$]  ${\overline{L}'_r}_{\mid ]0,+\infty[}$ is $]-\infty,+\infty]$-valued, non-decreasing  and right continuous. 
 \end{itemize}
When  the above holds and  $\lambda>0$,   we have    ${\overline{L}'_r}(\lambda)\in\R$   if and only if  $\overline{L}(t)<+\infty$ for some $t>\lambda$.
\end{lemma}

\begin{proof}
$(ii)\Rightarrow(i)$ follows from the definitions since $\overline{L}'_r$ takes the value $-\infty$  when  $\overline{L}_{\mid [0,+\infty[}$ is improper (\cf Lemma \ref{existence L on t}).
Assume that $(i)$ holds.
Since $\overline{L}$ is convex and $\overline{L}(0)=0$, the map ${\overline{L}'_r}_{\mid ]0,+\infty[}$ is $]-\infty,+\infty]$-valued. Let $\lambda>0$. 
If $\overline{L}'_r(\lambda)=+\infty$, then $\lambda$ fulfils one of the last two cases of Lemma \ref{existence L on t},  so that 
 $\overline{L}'_r$ takes only  the value  $+\infty$ on $[\lambda,+\infty[$ hence is  non-decreasing and right continuous on
 $[\lambda,+\infty[$.
Assume that $\overline{L}'_r(\lambda)<+\infty$. Then $\lambda$  fulfils one of the first three cases of 
Lemma \ref{existence L on t}; in particular, there exists  $\varepsilon>0$ such that $\overline{L}_{\mid [\lambda,\lambda+\varepsilon]}$ is real-valued and continuous.
Extend $\overline{L}_{\mid [\lambda,\lambda+\varepsilon]}$  to a lower semi-continuous function $f$  on $\R$ by putting $f(x)=+\infty$ for all $x\in\R\setminus[\lambda,\lambda+\varepsilon]$. The right derivative map    of $f$ is non-decreasing and right continuous on $\R$ by  
  Theorem 24.1 of  \cite{Rockafellar-70}  so that ${\overline{L}'_r}$ is non-decreasing on $[\lambda,\lambda+\varepsilon]$ and right continuous on $[\lambda,\lambda+\varepsilon[$; 
  therefore, $(ii)$ holds and the first assertion is proved.  If      $\overline{L}(t)<+\infty$ for some $t>\lambda$, then  ${\overline{L}'_r}(\lambda)<+\infty$ (by definition) hence  ${\overline{L}'_r}(\lambda)\in\R$ since $\lambda>0$ and  $\overline{L}_{\mid [0,+\infty[}$ is proper; the converse is obvious.
\end{proof}

\begin{lemma}\label{right-continuity of right derivative}
  The following statements are equivalent:
  \begin{itemize}
  \item[$(i)$] $\overline{L}'_r$ is not right continuous  at $0$;
  \item[$(ii)$] $\overline{L}_{\mid [0,+\infty[}$ is not lower semi-continuous  at  $0$, $\overline{L}'_r(0)=-\infty$ and  $\overline{L}_r'(0^+)>-\infty$.
   \item[$(iii)$] $\overline{L}_{\mid [0,+\infty[}$ is not lower semi-continuous  at  $0$ and   $\overline{L}_r'(0^+)\in\R$.
      \item[$(iv)$] $\overline{L}'_r(0)=-\infty$ and  $\overline{L}_r'(0^+)>-\infty$. 
  \end{itemize}
  The above equivalence hold verbatim replacing $\overline{L}_r'(0^+)>-\infty$ by $\overline{L}_r'(0^+)\in\R$. In particular,  $\overline{L}'_r(0)\in\R$ if and only if $\overline{L}'_r(0)=\overline{L}'_r(0^+)\in\R$.
           \end{lemma}
\begin{proof}
Assume that $(i)$ holds. By Lemma \ref{existence L on t}$a)$ the map  $\overline{L}_{\mid [0,+\infty[}$ is proper, and by Lemma  \ref{existence L on t}$b)$  there exists $\varepsilon>0$ such that $\overline{L}_{\mid [0,\varepsilon]}$ is real-valued and $\overline{L}_{\mid ]0,\varepsilon]}$ is continuous.
  If  $\overline{L}_{\mid [0,+\infty[}$ is  lower semi-continuous  at $0$, then the extension of   $\overline{L}_{\mid [0,\varepsilon]}$ to a lower semi-continuous function   on $\R$ yields the right continuity of $\overline{L}'_r$ at $0$ (by Theorem 24.1 of  \cite{Rockafellar-70}), which gives   a contradiction; therefore, $\overline{L}_{\mid [0,+\infty[}$ is not lower semi-continuous  at  $0$; since $\overline{L}$ is convex and $\overline{L}(0)=0$  it follows  that $\overline{L}(0^+)$ exists as a negative number, and consequently,   $\overline{L}_r'(0)=-\infty$.  Since $\overline{L}'_r(t)\in\R$ for all $t\in\ ]0,\varepsilon[$,  and  ${\overline{L}'_r}_{\mid ]0,+\infty[}$ is non-decreasing by Lemma \ref{improper-L}, $\overline{L}_r'(0^+)$ exists in $[-\infty,+\infty[$, hence  $\overline{L}_r'(0^+)>-\infty$ (since $\overline{L}'_r$ is not right continuous  at $0$ by hypothesis) and $(ii)$ holds. 
  
  Assume that $(ii)$ holds. Then,  $\overline{L}_r'(0^+)\in\R$ (otherwise  $\overline{L}_r'(0^+)=+\infty$ would imply $\overline{L}_r'(0)=+\infty$ and a contradiction) hence $(iii)$ holds. 
  
   Assume that $(iii)$ holds. By Lemma  \ref{existence L on t}  there exists $\varepsilon>0$ such that $\overline{L}_{\mid [0,\varepsilon]}$ is real-valued and $\overline{L}_{\mid ]0,\varepsilon]}$ is continuous, 
 hence (since $\overline{L}$ is convex and $\overline{L}(0)=0$),  $\overline{L}(0^+)$ exists as a negative number, which implies $\overline{L}_r'(0)=-\infty$, and gives $(iv)$.  Since the implication  $(iv)\Rightarrow (i)$ is obvious, the proof  of the first two assertions is complete; the last assertion is a direct consequence of them.
    \end{proof}

 Let $l_0$ be the function defined on $\R$ by
 \[\forall x\in\R,\ \ \ \ \ \ l_0(x)=-\lim_{\varepsilon\rightarrow 0,\varepsilon>0} \limsup_\alpha c_\alpha\log\mu_\alpha (]x-\varepsilon,x+\varepsilon[);\]
 note that $l_0$ is $[0,+\infty]$-valued and  lower semi-continuous.
 
 \begin{lemma}\label{l0-greater-L*}
We have 
\[\forall x\in\R,\ \ \ \ \ \ \ l_0(x)\ge\overline{L}^*(x).\]
\end{lemma}

\begin{proof}
Since for each  real number $\lambda\neq 0$ the set $\{x\in\R:a\le e^{\lambda x}\le b\}$ is compact for all $(a,b)\in\R^2$ with $a\le b$, Theorem 1 of \cite{Comman(2007)STAPRO77} yields 
 \[\forall\lambda\in\R\setminus\{0\},\ \ \ \ \ \overline{L}(\lambda)\ge\sup_{y\in \R}\{\lambda y-l_0(y)\}.\]
 Since $l_0$ is a $[0,+\infty]$-valued function and $\overline{L}(0)=0$,  the above inequality is  true  with $\lambda=0$ so that 
\[\forall (\lambda,x)\in\R^2,\ \ \ \ \ \overline{L}(\lambda)-\lambda x\ge\sup_{y\in \R}\{\lambda y-l_0(y)\}-\lambda x\ge-l_0(x),\]
hence 
\[\forall x\in\R,\ \ \ \ \ \ \ l_0(x)\ge \sup_{\lambda\in\R}\{\lambda x-\overline{L}(\lambda)=\overline{L}^*(x).\]
\end{proof}

 \begin{lemma}\label{expression-L-when-locally-finite}
  Assume that  $\overline{L}_{\mid [0,+\infty[}$ is proper and  $\overline{L}(t)<+\infty$ for some $t>\lambda$.  There exists $\varepsilon>0$ such that  for each $t\in\ ]\lambda,\lambda+\varepsilon[$ we have 
 \[\overline{L}(t)=\sup_{x\in\R}\{tx-l_0(x)\}=tx_t-l_0(x_t)\]
 for some $x_t\in\R$.
 \end{lemma}

\begin{proof}
By Lemma \ref{existence L on t} there exists $\varepsilon>0$ such that 
 $\overline{L}_{\mid [\lambda,\lambda+\varepsilon]}$  is real-valued, hence  
\[\forall t\in\ ]\lambda,\lambda+\varepsilon[,\ \ \ \ \ \ \lim_{M\rightarrow+\infty}\limsup_\alpha c_\alpha\log\int_{[M/t,+\infty[} e^{c_\alpha^{-1}t z}\mu_\alpha(dz)=-\infty\]  (Lemma 4.3.8 of  \cite{Dembo-Zeitouni}). 
 Part (b) of  Theorem 1 of \cite{Comman(2007)STAPRO77} yields for each $t\in\ ]\lambda,\lambda+\varepsilon]$ and for each 
 $M$ large enough
   \begin{equation}\label{expression-L-when-locally-finite-eq-10}
\overline{L}(t)=\sup_{x\le M/t}\{tx-l_0(x)\}.
\end{equation}
  Let $t\in\ ]\lambda,\lambda+\varepsilon[$. For each integer $n\ge 1$ there exists $x_n\le M/t$  such that 
  \begin{equation}\label{expression-L-when-locally-finite-eq-20}
  \overline{L}(t)\ge tx_n-l_0(x_n)>\overline{L}(t)-1/n\ge\overline{L}(t)-1
  \end{equation}
  hence $x_n\in\left[\overline{L}(t)/t-1/t,M/t\right]$. Therefore, the sequence $(x_n)$ has a  subsequence $(x_{n_m})$  converging to some $x\in\left[\overline{L}(t)/t-1/t,M/t\right]$,  so that   letting $n\rightarrow+\infty$ in (\ref{expression-L-when-locally-finite-eq-20}) yields
     \[
  \overline{L}(t)=\lim_m(tx_{n_m}-l_0(x_{n_m}))=tx-\lim_m l_0(x_{n_m})\le tx-l_0(x),
  \]
   where the last inequality follows from the lower semi-continuity of $l_0$.
   From the above expression  and  (\ref{expression-L-when-locally-finite-eq-10}),  we get  $\overline{L}(t)=tx-l_0(x)$, which proves the lemma. 
\end{proof}

\begin{lemma}\label{l0=l1=L*-on-left-right derivative}
  Assume that  $\overline{L}_{\mid [0,+\infty[}$ is proper and  $\overline{L}(t)<+\infty$ for some $t>\lambda$. There exists $\varepsilon>0$ such that  
 \[\forall t\in\  ]\lambda,\lambda+\varepsilon[,\ \ \ \ \ \ \ \ l_0(\overline{L}_r'(t))=\overline{L}^*(\overline{L}_r'(t)).\]
 When $\overline{L}'_r(\lambda^+)\neq-\infty$, the above equality is true with $t=\lambda$,  and $\overline{L}_r'(\lambda^+)$ in place of  $\overline{L}_r'(t)$.
 \end{lemma}

\begin{proof}
By Lemma \ref{existence L on t}
 there exists $\varepsilon_0>0$ such that 
 $\overline{L}_{\mid [\lambda,\lambda+\varepsilon_0]}$  is real-valued bounded and $\overline{L}_{\mid ]\lambda,\lambda+\varepsilon_0]}$  continuous. Let $t\in\ ]\lambda,\lambda+\varepsilon_0[$. 
For each $\varepsilon\in[0,\lambda+\varepsilon_0-t[$ let $\partial\overline{L}(t+\varepsilon)$ denote the set of subgradients of $\overline{L}$ at $t+\varepsilon$. Note that 
\[\sup\partial\overline{L}(t)=\overline{L}'_r(t)\in\R\]  and \[\forall \varepsilon\in\ ]0,\lambda+\varepsilon_0-t[,\ \ \ \ \ \ \partial \overline{L}(t+\varepsilon)=\left[\overline{L}'_l(t+\varepsilon),\overline{L}'_r(t+\varepsilon)\right].\] 
 For each $\varepsilon\in\ ]0,\lambda+\varepsilon_0-t[$,  Lemma \ref{expression-L-when-locally-finite} yields some  $x_{t+\varepsilon}\in\R$ such that
  \[\overline{L}(t+\varepsilon)=(t+\varepsilon) x_{t+\varepsilon}-l_0(x_{t+\varepsilon})\]
  hence by Lemma \ref{l0-greater-L*},
  \[\overline{L}(t+\varepsilon)\le (t+\varepsilon) x_{t+\varepsilon}-\overline{L}^*(x_{t+\varepsilon})\le\overline{L}^{**}(t+\varepsilon);\]
  since $\overline{L}\ge\overline{L}^{**}$  we obtain 
 $l_0(x_{t+\varepsilon})=\overline{L}^*(x_{t+\varepsilon})$; 
in particular, $x_{t+\varepsilon}\in\partial\overline{L}(t+\varepsilon)$. 
  Putting $x=\overline{L}_r'(t)$ we have 
   \[
  \forall\varepsilon\in\  ]0,\lambda+\varepsilon_0-t[,\ \ \ \ \ x=\sup\partial\overline{L}(t)\le x_{t+\varepsilon}\le\overline{L}'_r(t+\varepsilon)
  \]
    hence $\lim_{\varepsilon\rightarrow 0} x_{t+\varepsilon}=x$ 
   (because $\lim_{\varepsilon\rightarrow 0}\overline{L}'_r(t+\varepsilon)=x$ by Lemma \ref{improper-L}) and 
 \begin{multline*}\label{l0=l1=L*-on-left-right derivative-eq100}
 x=\lim_{\varepsilon\rightarrow
0^+}\frac{\overline{L}(t+\varepsilon)-\overline{L}(t)}{\varepsilon}=
\lim_{\varepsilon\rightarrow 0^+}\frac{(t+\varepsilon)x_{t+\varepsilon}-l_0(x_{t+\varepsilon})-
\overline{L}(t)}{\varepsilon}
\\
= x+\lim_{\varepsilon\rightarrow 0^+}\frac{t x_{t+\varepsilon}-l_0(x_{t+\varepsilon})-\overline{L}(t)}{\varepsilon},
 \end{multline*}
 which implies
 \[
 0=\lim_{\varepsilon\rightarrow 0^+}(t x_{t+\varepsilon}-l_0(x_{t+\varepsilon})-\overline{L}(t))
\le tx-l_0(x)-\overline{L}(t)\le 0,
 \]
 where the first inequality follows from the lower semi-continuity of $l_0$, and 
 the second inequality follows from Lemma \ref{expression-L-when-locally-finite}; 
therefore, $\overline{L}(t)=tx-l_0(x)$ hence $l_0(x)=\overline{L}^*(x)$ 
 since  $\overline{L}(t)=tx-\overline{L}^*(x)$ (because $x\in\partial\overline{L}(t)$); this proves  the first assertion. 
 
 When $\overline{L}'_r(\lambda^+)\neq-\infty$, the hypotheses imply $\overline{L}'_r(\lambda^+)\in\R$ and the last assertion follows noting that the above proof works verbatim with $t=\lambda$ and $\overline{L}_r'(\lambda^+)$ (resp. $\overline{L}(\lambda^+)$) in place of $\overline{L}_r'(t)$ (resp. $\overline{L}(t)$).
 \end{proof}

\begin{lemma}\label{inf L*}
Assume that  $\overline{L}_{\mid [0,+\infty[}$ is proper and  $\overline{L}(t)<+\infty$ for some $t>0$. Let 
   $(t_i)$ be a sequence   of positive numbers converging to $0$.
  We have   
   \[\lim_i\overline{L}^*(\overline{L}'_r(t_i))=-\overline{L}(0^+)\in [0,+\infty[.\]
   If furthermore $\overline{L}'_r(0^+)>-\infty$, then 
   \[\lim_i\overline{L}^*(\overline{L}'_r(t_i))=\overline{L}^*(\overline{L}'_r(0^+))=\sup_{t>0}\{t \overline{L}'_r(0^+)-\overline{L}(t)\}.\]
    \end{lemma}

\begin{proof}
By Lemma \ref{existence L on t} (applied with $\lambda=0$),  
  there exists $\varepsilon>0$ such that $\overline{L}_{\mid ]0,\varepsilon]}$  is real-valued, bounded  and continuous, hence  eventually 
  $\overline{L}'_r(t_i)\in\R$, 
  \begin{equation}\label{inf L*-eq20}
\overline{L}^*(\overline{L}'_r(t_i))=t_i \overline{L}'_r(t_i)-\overline{L}(t_i)=\sup_{t> 0}\{t \overline{L}'_r(t_i)-\overline{L}(t)\}, 
\end{equation}
  and $\overline{L}(0^+)$ exists in $]-\infty,0]$; furthermore, $\overline{L}'_r(0^+)$ exists in $[-\infty,+\infty[$ by 
   Lemma \ref{improper-L}. Since   $\overline{L}^*$ is continuous  on its effective domain,    and 
   $\overline{L}(0^+)$ is a non-positive real number,  (\ref{inf L*-eq20})  ensures the existence of 
     $\lim_i \overline{L}^*(\overline{L}'_r(t_i))$   in   $[0,+\infty[$. 
     
     First assume that   $\overline{L}'_r(0^+)>-\infty$. Then  
   $\overline{L}'_r(0^+)\in\R$  and   (\ref{inf L*-eq20}) yields
\[\overline{L}(0^+)=\lim_i \overline{L}(t_i)=\lim_i (t_i \overline{L}'_r(t_i)-\overline{L}^*(\overline{L}'_r(t_i)))=-\lim_i \overline{L}^*(\overline{L}'_r(t_i))=-\overline{L}^*(\overline{L}'_r(0^+)),\]
which proves the first assertion and the first equality of the second assertion;  we have 
\[\overline{L}^*(\overline{L}'_r(0^+))\ge \sup_{t> 0}\{t \overline{L}'_r(0^+)-\overline{L}(t)\}\ge \lim_i\{t_i \overline{L}'_r(0^+)-\overline{L}(t_i)\}=-\overline{L}(0^+),\]
and the second equality of the second   assertion follows.

     Assume   that  $\overline{L}'_r(0^+)=-\infty$. 
    For each $i$ large enough   there exists     $\lambda_i>0$ such that  
     $\lambda_i \overline{L}'_r(t_i)=\overline{L}(\lambda_i)$, hence 
      \[
      \sup_{t>\lambda_i}\{t\overline{L}'_r(t_i)-\overline{L}(t)\}= 0
      \]
      and 
   \[
\lim_i\sup_{t\in\  ]0,\lambda_i]}\{t \overline{L}'_r(t_i) -\overline{L}(t)\}=-\overline{L}(0^+), 
\]
           so that the first assertion  follows from (\ref{inf L*-eq20}) since 
                                         $-\overline{L}(0^+)\ge 0$.
                   \end{proof}

Lemma \ref{inf L*} shows that when $\overline{L}'_r(0^+)=-\infty$,    the map $\overline{L}^*$  extends by continuity to a 
$[-\overline{L}(0^+),+\infty]$-valued map on $[-\infty,+\infty[$ 
 by putting  $\overline{L}^*(-\infty)=-\overline{L}(0^+)$; in what follows,  we  implicitely use this extension.

\begin{lemma}\label{lemma-case-differentiable}
We assume that   $\overline{L}_{\mid ]\lambda,\lambda+\varepsilon[}$ is    differentiable for some $\varepsilon>0$. 
\begin{nitemize}
\item[$a)$]  ${\overline{L}'}_{\mid  ]\lambda,\lambda+\varepsilon[}$  extends   to a non-decreasing 
  continuous   surjection between   $[\lambda,\lambda+\varepsilon]$ and  $\left[\overline{L}_r'(\lambda^+),\overline{L}_l'(\lambda+\varepsilon)\right]$.
\item[$b)$] ${\overline{L}^*}_{\mid [\overline{L}_r'(\lambda^+), \overline{L}_l'(\lambda+\varepsilon)[}$ is  $[0,+\infty[$-valued and 
extends to a   $[0,+\infty]$-valued,  strictly increasing and  continuous map on 
 $\left[\overline{L}_r'(\lambda^+), \overline{L}_l'(\lambda+\varepsilon)\right]$, which takes the value $+\infty$ at $\overline{L}_l'(\lambda+\varepsilon)$ if and only if $\overline{L}_l'(\lambda+\varepsilon)=+\infty$; 
  it  vanishes at $\overline{L}_r'(\lambda^+)$ if and only if either $\lambda>0$, $\overline{L}$ is differentiable at $\lambda$  and  linear on $[0,\lambda]$, or 
$\lambda=0$ and $\overline{L}$  is right continuous at zero. The map  ${\overline{L}^*}_{\mid ]\overline{L}_r'(\lambda^+), \overline{L}_l'(\lambda+\varepsilon)[}$ is positive and 
  strictly convex; if furthermore,  ${\tilde{\lambda}}<\lambda+\varepsilon$ and $\overline{L}_{\mid ]{\tilde{\lambda}},\lambda+\varepsilon[}$ is strictly convex, then  ${\overline{L}^*}_{\mid ]\overline{L}_r'(\lambda^+), \overline{L}_l'(\lambda+\varepsilon)[}$  is differentiable (with  ${\tilde{\lambda}}=\sup\{t>\lambda: \overline{L}_{\mid ]\lambda,t]}\ \textit{is affine}\}$). 
\end{nitemize}
\end{lemma}

\begin{proof}
$a)$ Since $\overline{L}_{\mid ]\lambda,\lambda+\varepsilon[}$ is     differentiable,
 the map ${\overline{L}'}_{\mid  ]\lambda,\lambda+\varepsilon[}$ is continuous  by  Corollary  25.5.1 of \cite{Rockafellar-70}. Since $\overline{L}'_l(\lambda+\varepsilon)=\lim_{t\rightarrow\varepsilon,t<\varepsilon}\overline{L}'(\lambda+t)$,  the map  ${\overline{L}'}_{\mid  ]\lambda,\lambda+\varepsilon[}$ is a non-decreasing 
  continuous   surjection  onto  $\left]\overline{L}_r'(\lambda^+),\overline{L}_l'(\lambda+\varepsilon)\right[$, which  can be extended by continuity  to a non-decreasing 
  continuous   surjection between   $[\lambda,\lambda+\varepsilon]$ and  $\left[\overline{L}_r'(\lambda^+),\overline{L}_l'(\lambda+\varepsilon)\right]$.

 $b)$   Each $t\in\ ]\lambda,\lambda+\varepsilon[$ is a subgradient of  $\overline{L}^*$ at $\overline{L}'(t)$ so that ${\overline{L}^*}_{\mid ]\overline{L}_r'(\lambda^+), \overline{L}_l'(\lambda+\varepsilon)[}$ is  $[0,+\infty[$-valued and strictly increasing, hence 
 ${\overline{L}^*}_{\mid [\overline{L}_r'(\lambda^+), \overline{L}_l'(\lambda+\varepsilon)[}$  is  $[0,+\infty[$-valued and strictly increasing; in particular,   ${\overline{L}^*}_{\mid ]\overline{L}_r'(\lambda^+), \overline{L}_l'(\lambda+\varepsilon)[}$ is positive. 
 Since $\overline{L}^*$ is lower semi-continuous, it is continuous on its effective domain so that ${\overline{L}^*}_{\mid ]\overline{L}_r'(\lambda^+), \overline{L}_l'(\lambda+\varepsilon)[}$ is     continuous,  and thus  extends  to a  $[0,+\infty]$-valued,   strictly increasing and continuous   map on 
 $\left[\overline{L}_r'(\lambda^+), \overline{L}_l'(\lambda+\varepsilon)\right]$; this extended map can takes the value $+\infty$ only at $\overline{L}_l'(\lambda+\varepsilon)$; since $\overline{L}^*$ is lower semi-continuous, 
 $\overline{L}^*(\overline{L}_l'(\lambda+\varepsilon))$ is finite when $\overline{L}_l'(\lambda+\varepsilon)$ is finite; conversely, if 
 $\overline{L}_l'(\lambda+\varepsilon)=+\infty$, then $\left]\overline{L}_r'(\lambda^+),+\infty\right[$ is included in the effective domain of $\overline{L}^*$,  hence $\lim_{t\rightarrow\lambda+\varepsilon,t<\lambda+\varepsilon}\overline{L}^*(\overline{L}_r'(t))=+\infty$.  
 Let $g$ be the extension of   $\overline{L}_{\mid ]\lambda,\lambda+\varepsilon[}$ to $\R$ defined by putting
    \[
g(t) = \left\{\begin{array}{ll}
 +\infty & \  \textnormal{if $t\in\ ]-\infty,\lambda[$\ and $\overline{L}'_r(\lambda^+)=-\infty$}
\\
 (t-\lambda)\overline{L}'_r(\lambda) +  g(\lambda) &\ \textnormal{if  $t\in\ ]-\infty,\lambda[$ and  $\overline{L}'_r(\lambda^+)>-\infty$}
  \\
\overline{L}(\lambda^+) &\  \textnormal{if $t=\lambda$}
\\
 \overline{L}(t) &\  \textnormal{if $t\in\ ]\lambda,\lambda+\varepsilon[$}
 \\
  \overline{L}((\lambda+\varepsilon)^-) &\  \textnormal{if $t=\lambda+\varepsilon$}
    \\
 (t-(\lambda+\varepsilon))\overline{L}'_l(\lambda) + g(\lambda+\varepsilon) &\  \textnormal{if $t\in\ ]\lambda+\varepsilon,+\infty[$ and  $\overline{L}'_l(\lambda+\varepsilon)<+\infty$}
 \\
 +\infty &\ \textnormal{if  $t\in\ ]\lambda+\varepsilon,+\infty[$ and  $\overline{L}'_l(\lambda+\varepsilon)=+\infty$},
      \end{array}
    \right.
\]
so  that  $g$ is   convex,  lower semi-continuous,   differentiable on the interior of its effective domain, but not sub-differentiable at each point in  the complement of the interior of its effective domain; therefore, the  Legendre-Fenchel transform $g^*$ of $g$ 
  is strictly convex   on  the interior of its effective domain (\cite{Rockafellar_Wets}, Theorem 11.13), hence on 
 $]g_r'(\lambda^+),g_l'(\lambda+\varepsilon)[$. 
 Since $g'(t)=\overline{L}'(t)$ for all $t\in\ ]\lambda,\lambda+\varepsilon[$, we obtain 
   ${g^*}_{\mid ]g_r'(\lambda^+),g_l'(\lambda+\varepsilon)[}=\overline{L}^*_{\mid ]\overline{L}_r'(\lambda^+), \overline{L}_l'(\lambda+\varepsilon)[}$, which proves the strict convexity property. 
  The  first part of the first assertion, and  the first part of the second assertion are proved.   
 The continuity and strict increasingness  imply  that 
 $\inf_{ [\overline{L}'_r(\lambda^+),\overline{L}_r'(\lambda+\varepsilon)[} \overline{L}^*$ is a minimum, which  is attained at the unique point $\overline{L}'_r(\lambda^+)$; this proves the second part of the first  assertion  since $ \overline{L}^*(\overline{L}'_r(\lambda^+))=\lambda\overline{L}_r'(\lambda^+)-\overline{L}(\lambda^+)$ (using Lemma \ref{inf L*} when $\lambda=0$).
 Assume furthermore that   $\overline{L}_{\mid ]\lambda,\lambda+\varepsilon[}$ is strictly convex. The map 
 $g$ is not sub-differentiable at some $t\in[\lambda,\lambda+\varepsilon]$ if and only if either $t=\lambda=0$ and  $g'_r(0^+)=-\infty$, or   $t=\lambda+\varepsilon$ and  $g'_l(\lambda+\varepsilon)=+\infty$, hence  
 $]\lambda,\lambda+\varepsilon[\ \subset\{t\in[\lambda,\lambda+\varepsilon]:g\  \textnormal{is sub-differentiable at $t$}\}$.  
  Therefore, the map $g 1_{[\lambda,\lambda+\varepsilon]}+\infty1_{\R\setminus [\lambda,\lambda+\varepsilon]}$  is  strictly   convex on  every convex subset of the set $\{t\in[\lambda,\lambda+\varepsilon]: g 1_{[\lambda,\lambda+\varepsilon]}+\infty 1_{\R\setminus [\lambda,\lambda+\varepsilon]}\  \textnormal{is sub-differentiable at $t$}\}$, hence its  Legendre-Fenchel transform  is differentiable    on  the interior of its effective domain   (\cite{Rockafellar_Wets}, Theorem 11.13); consequently,   $\overline{L}^*_{\mid ]\overline{L}_r'(\lambda^+), \overline{L}_l'(\lambda+\varepsilon)[}$ is strictly convex. 
    Assume that ${\tilde{\lambda}}<\lambda+\varepsilon$ and $\overline{L}_{\mid ]{\tilde{\lambda}},\lambda+\varepsilon[}$ is strictly convex. The above proof and  the proof of part $a)$   work verbatim replacing $]\lambda,\lambda+\varepsilon[$ by $]{\tilde{\lambda}},\lambda+\varepsilon[$.  
  Since $\overline{L}$ is differentiable at ${\tilde{\lambda}}$, we have $\overline{L}'_r(\lambda^+)=\overline{L}'_r({\tilde{\lambda}}^+)$, which proves 
  the second part of the second assertion.
   \end{proof}

\begin{proof-theo} 
Put $x=\overline{L}_r'(\lambda^+)$. 
The hypotheses imply  that $\overline{L}_{\mid [0,+\infty[}$  is proper 
(otherwise, by Lemma \ref{existence L on t}$a)$,  $\overline{L}'_r(t)\in\{-\infty,+\infty\}$ for all $t>0$,  which contradicts the hypothesis when $x>-\infty$, and $\overline{L}_{\mid ]0,t[}=-\infty$ for all $t>0$ small enough,  which contradicts the hypothesis when $x=-\infty$),  and  the case $(i)$ of Lemma \ref{existence L on t}$b)$  holds; in particular, $x<+\infty$ and there exists $\varepsilon>0$ such that  $\overline{L}_{\mid [\lambda,\lambda+\varepsilon[}$ is  real-valued  and $\overline{L}_{\mid ]\lambda,\lambda+\varepsilon]}$ is continuous. 
       Let  $(x_\alpha, y_\alpha)$ be a net  in $[-\infty,+\infty]^2$  such that $\lim_\alpha x_\alpha=x$ and $\liminf_\alpha y_\alpha>x$. Put $y=\liminf_\alpha y_\alpha$.

\bigskip
      
 $\bullet$ First   assertion,  the case  $x\in\R$:  For each $t\in\ ]\lambda,\lambda+\varepsilon[$, 
 Chebyshev's inequality  yields eventually
\[
   c_\alpha\log\int_{\R} e^{c_\alpha^{-1} t z}\mu_\alpha(dz)
  \ge 
  c_\alpha\log( e^{c_\alpha^{-1} t x_\alpha}\mu_\alpha([x_\alpha,y_\alpha]))
=t x_\alpha + c_\alpha\log\mu_\alpha([x_\alpha,y_\alpha])
\]
hence
\[
   \limsup_\alpha c_\alpha\log\mu_\alpha([x_\alpha,y_\alpha])\le \overline{L}(t)-tx,\]
 and letting $t\rightarrow\lambda$,
 \begin{equation}\label{end-proof-eq5}
\limsup_\alpha c_\alpha\log\mu_\alpha([x_\alpha,y_\alpha])\le\overline{L}(\lambda^+)-\lambda x=-\overline{L}^*(x),
 \end{equation}
 where the last equality follows from   Lemma \ref{inf L*}  when $\lambda=0$.
 Suppose that  
 \[
 \limsup_\alpha c_\alpha\log\mu_\alpha(]x_\alpha,y_\alpha[)< -\overline{L}^*(x).
 \]
    The hypothesis together with    the continuity of $\overline{L}^*$ on its effective domain implies the existence of      $t>\lambda$  and $\delta>0$   such that  
       \[x+2\delta<\overline{L}'_r(t)<y-2\delta\] and 
   \begin{equation}\label{end-proof-eq7}
   \limsup_\alpha c_\alpha\log\mu_\alpha(]x_\alpha,y_\alpha[)< -\overline{L}^*(\overline{L}'_r(t)).
   \end{equation}
 Since  eventually   \[x_\alpha+\delta<\overline{L}'_r(t)<y_\alpha-\delta,\]
     we have  eventually   
     \[\mu_\alpha(]x_\alpha,y_\alpha[)\ge \mu_\alpha(]\overline{L}'_r(t)-\delta,\overline{L}'_r(t)+\delta[)\]  hence 
    \[\limsup_\alpha c_\alpha\log\mu_\alpha(]x_\alpha,y_\alpha[\ge -l_0(\overline{L}'_r(t)),\]
    which  contradicts (\ref{end-proof-eq7}) since $l_0(\overline{L}'_r(t))=\overline{L}^*(\overline{L}'_r(t))$ by Lemma \ref{l0=l1=L*-on-left-right derivative}; therefore, we have
\[\limsup_\alpha c_\alpha\log\mu_\alpha(]x_\alpha,y_\alpha[)\ge -\overline{L}^*(x),\]
which together with (\ref{end-proof-eq5}) proves the first three  equalities of the first assertion; the last equality is obvious when $\lambda>0$ (definition of $\overline{L}^*$), and  follows from  Lemma \ref{inf L*} when $\lambda=0$.

\bigskip

$\bullet$ First assertion,   the case  $x=-\infty$:   Lemma \ref{improper-L} implies $\lambda=0$.   Let $(t_i)$ be a sequence of positive numbers converging to $0$, so that 
    eventually 
\[l_0(\overline{L}'_r(t_i))=\overline{L}^*(\overline{L}'_r(t_i))\]  
 by Lemma \ref{l0=l1=L*-on-left-right derivative}.
 Since  $\lim_i\overline{L}'_r(t_i)=-\infty$,  there exists $\delta>0$ such that  $x_\alpha+\delta<\overline{L}'_r(t_i)<y_\alpha-\delta$  eventually with respect to $i$  and eventually with respect to $\alpha$, which together with the above equality yields
  \begin{multline*}
  0\ge\limsup_\alpha c_\alpha\log\mu_\alpha([x_\alpha,y_\alpha])\ge
\limsup_\alpha c_\alpha\log\mu_\alpha(]x_\alpha,y_\alpha[)\ge\limsup_i (-l_0(\overline{L}'_r(t_i))
\\
=\limsup_i -(\overline{L}^*(\overline{L}'_r(t_i)))
=-\lim_i\overline{L}^*(\overline{L}'_r(t_i))=\overline{L}(0^+)=0,
 \end{multline*} 
 where the third equality is given by Lemma \ref{inf L*}, and the last equality follows from the hypothesis of 
   right continuity of 
$\overline{L}$  at zero.   The first two equalities of the first assertion  follow from the above expression  together with Lemma \ref{inf L*} (recall that by convention, $0\cdot (-\infty)=0$).
  For 
 each $i\in\N$,   $t_i$ is a subgradient  of $\overline{L}^*$ at $\overline{L}'_r(t_i)$ so that  $\overline{L}^*$ is non-decreasing on $\left]-\infty, \overline{L}'_r(t_i)\right]$; 
 since $\lim_\alpha x_\alpha=-\infty$,   eventually   $x_\alpha$ belongs to the effective domain of $\overline{L}^*$ and fullfils 
  \[0\le\overline{L}^*(x_\alpha)\le\overline{L}^*(\overline{L}'_r(t_i)),\]
  hence 
  \[0\le\liminf_\alpha\overline{L}^*(x_\alpha)\le\limsup_\alpha\overline{L}^*(x_\alpha)\le\overline{L}^*(\overline{L}'_r(t_i));\] 
letting $i\rightarrow+\infty$ gives $\lim_\alpha\overline{L}^*(x_\alpha)=0$, which proves the last two equalities of the first assertion.
The proof of the first assertion is complete.

\bigskip

$\bullet$ Second assertion: Let $(t_i)$ be a  sequence  in $]{\tilde{\lambda}},+\infty[$  converging  to ${\tilde{\lambda}}$ such that 
   $L(t_i)$ exists for all $i\in\N$. Let $(\mu_\beta,c_\beta,x_\beta,y_\beta)$ be a subnet of  $(\mu_\alpha,c_\alpha,x_\alpha,y_\alpha)$. 
   For each $t\in\R$ we put
\[\overline{L}^{(\mu_\beta,c_\beta)}(t)=\limsup_\beta c_\beta\log\int_\R e^{c_\beta^{-1} t x}\mu_\beta(dx).\]
Since  $\overline{L}^{(\mu_\beta,c_\beta)}(t_i)=L(t_i)$ for all $i\in\N$,
    we have 
 \begin{equation}\label{end-proof-eq20}
 \overline{L}^{(\mu_\beta,c_\beta)}({\tilde{\lambda}}^+)=\lim_i \overline{L}^{(\mu_\beta,c_\beta)}(t_i)=\lim_i\overline{L}(t_i)=\overline{L}({\tilde{\lambda}}^+)
 \end{equation}
  hence 
    \begin{equation}\label{end-proof-eq30}
{\overline{L}^{(\mu_\beta,c_\beta)}}'_r({\tilde{\lambda}}^+)
 =\lim_{i}\frac{\overline{L}^{(\mu_\beta,c_\beta)}(t_i)-\overline{L}^{(\mu_\beta,c_\beta)}({\tilde{\lambda}}^+)}{t_i-\lambda}
 \\
 =
\lim_{i}\frac{\overline{L}(t_i)-\overline{L}({\tilde{\lambda}}^+)}{t_i-\lambda}=\overline{L}'_r({\tilde{\lambda}}^+)=x,
\end{equation}
where the last equality follow from  Lemma \ref{existence L on t}$b)$. The inequality ${\overline{L}^{(\mu_\beta,c_\beta)}}'_r (\lambda^+)\le{\overline{L}^{(\mu_\beta,c_\beta)}}'_r ({\tilde{\lambda}}^+)$ together with (\ref{end-proof-eq30}) implies
    \begin{equation}\label{end-proof-eq32}
{\overline{L}^{(\mu_\beta,c_\beta)}}'_r (\lambda^+)\le x.
\end{equation}
We have  $\overline{L}'_r(t_i)>\overline{L}'_r({\tilde{\lambda}}^+)$ for all $i\in\N$ (by definition of ${\tilde{\lambda}}$)  and  
    \begin{equation}\label{end-proof-eq42}
    x=\lim_i\overline{L}'_r(t_i),
     \end{equation} so that 
 $(\overline{L}'_r(t_i))$ has a strictly decreasing subsequence $(\overline{L}'_r(t_{i_j}))$ converging to $\overline{L}'_r({\tilde{\lambda}}^+)$, which implies eventually $t_{i_{j+1}}< t_{i_j}< t_{i_{j-1}}$, hence 
  \begin{equation}\label{end-proof-eq45}
x=\lim_j\overline{L}'_l(t_{i_j})
\end{equation}
  and  
  \begin{equation}\label{end-proof-eq47}
\forall j\in\N,\ \ \ \ \ \ \ \ \overline{L}'_l(t_{i_j})>x.
\end{equation}
Put  ${\tilde{\lambda}}^{(\mu_\beta,c_\beta)}=\sup\{t>\lambda: \overline{L}^{(\mu_\beta,c_\beta)}_{\mid ]\lambda,t]}\ \textnormal{is affine}\}$.
The inequality  $\overline{L}^{(\mu_\beta,c_\beta)}\le \overline{L}$ together with (\ref{end-proof-eq20})  implies
\begin{equation}\label{end-proof-eq48}
{\tilde{\lambda}}^{(\mu_\beta,c_\beta)}\ge {\tilde{\lambda}}.
\end{equation}

\begin{claim}\label{claim 1} 
When  ${\overline{L}^{(\mu_\beta,c_\beta)}}'_r(\lambda^+)>-\infty$, the     hypothesis  of Theorem \ref{strength-Plachky-Steinebach} holds with the net $(\mu_\beta,c_\beta)$ in place of $(\mu_\alpha,c_\alpha)$. Furthermore, we have 
$\overline{L}^{(\mu_\beta,c_\beta)}(\lambda^+)=\overline{L}(\lambda^+)$ and ${\overline{L}^{(\mu_\beta,c_\beta)}}'_r (\lambda^+)=x$.
\end{claim}

\begin{proof-claim 1}
 Assume  ${\overline{L}^{(\mu_\beta,c_\beta)}}'_r (\lambda^+)>-\infty$. Since   $x>-\infty$ by (\ref{end-proof-eq32}) we have 
\begin{multline*}
x({\tilde{\lambda}}-\lambda)=\overline{L}({\tilde{\lambda}}^+)-\overline{L}(\lambda^+)\le \overline{L}({\tilde{\lambda}}^+)-\overline{L}^{(\mu_\beta,c_\beta)}(\lambda^+)=\overline{L}^{(\mu_\beta,c_\beta)}({\tilde{\lambda}}^+)-\overline{L}^{(\mu_\beta,c_\beta)}(\lambda^+)
\\
\le {\overline{L}^{(\mu_\beta,c_\beta)}}'_l({\tilde{\lambda}}^+)({\tilde{\lambda}}-\lambda)\le{\overline{L}^{(\mu_\beta,c_\beta)}}'_r({\tilde{\lambda}}^+)({\tilde{\lambda}}-\lambda)=x({\tilde{\lambda}}-\lambda),
\end{multline*}
where the second equality follows from (\ref{end-proof-eq20}) and the last equality follows from (\ref{end-proof-eq30}); therefore,  
all the above inequalities are equalities, which gives 
\begin{equation}\label{end-proof-eq33}
\overline{L}^{(\mu_\beta,c_\beta)}(\lambda^+)=\overline{L}(\lambda^+),
\end{equation}
which together with (\ref{end-proof-eq20}) implies 
\begin{equation}\label{end-proof-eq35}
\overline{L}^{(\mu_\beta,c_\beta)}_{\mid ]\lambda,{\tilde{\lambda}}]}=\overline{L}_{\mid ]\lambda,{\tilde{\lambda}}]}.
\end{equation}
Since $\overline{L}^{(\mu_\beta,c_\beta)}\le\overline{L}$ with $\overline{L}$ convex, (\ref{end-proof-eq35}) implies
\begin{equation}\label{end-proof-eq50}
{\tilde{\lambda}}^{(\mu_\beta,c_\beta)}={\tilde{\lambda}}.
\end{equation}
We have 
\begin{multline*}
\lim_j \overline{L}'_l(t_{i_j})\le\liminf_j {\overline{L}^{(\mu_\beta,c_\beta)}}'_l(t_{i_j})\le\liminf_j {\overline{L}^{(\mu_\beta,c_\beta)}}'_r(t_{i_j})\le  \limsup_j {\overline{L}^{(\mu_\beta,c_\beta)}}'_r(t_{i_j})
\\
\le \lim_j \overline{L}'_r(t_{i_j}),
\end{multline*}
hence by  (\ref{end-proof-eq42}) and  (\ref{end-proof-eq45}), all the above inequalities are equalities, which together with  (\ref{end-proof-eq30})  yields
  \[\forall j\in\N,\ \ \ \ 
{\overline{L}^{(\mu_\beta,c_\beta)}}'_r(t_{i_j})\ge {\overline{L}^{(\mu_\beta,c_\beta)}}'_l(t_{i_j})\ge
  \overline{L}'_l(t_{i_j})>\lim_j {\overline{L}^{(\mu_\beta,c_\beta)}}'_r(t_{i_j})=x,
 \]
 where the strict inequality follows from   (\ref{end-proof-eq47}). 
 The above expression together with (\ref{end-proof-eq30}) and  (\ref{end-proof-eq50}) gives
 \begin{equation}\label{end-proof-eq57}
 \forall j\in\N,\ \ \ \ \  {\overline{L}^{(\mu_\beta,c_\beta)}}'_r(t_{i_j})>\lim_j {\overline{L}^{(\mu_\beta,c_\beta)}}'_r(t_{i_j})={\overline{L}^{(\mu_\beta,c_\beta)}}'_r\left({\tilde{\lambda}}^{(\mu_\beta,c_\beta)}\right),
 \end{equation}
 which proves the first assertion of the claim. 
 
 The first equality of the second assertion  is given by  (\ref{end-proof-eq33}); the second  equality of the second assertion is given by (\ref{end-proof-eq30}) when $\lambda={\tilde{\lambda}}$.  Assume   $\lambda<{\tilde{\lambda}}$. Then,  (\ref{end-proof-eq35}) and     (\ref{end-proof-eq50})  yield
   \begin{equation}\label{end-proof-eq60}
  {\overline{L}^{(\mu_\beta,c_\beta)}}'_r({\tilde{\lambda}}^+)={\overline{L}^{(\mu_\beta,c_\beta)}}'_r (\lambda^+).
   \end{equation}
   Since $\overline{L}$ is differentiable at ${\tilde{\lambda}}$ when ${\tilde{\lambda}}>\lambda$, we have
   \[{\overline{L}^{(\mu_\beta,c_\beta)}}'_r (\lambda^+)\le x=
   \overline{L}'_l ({\tilde{\lambda}}^+)\le{\overline{L}^{(\mu_\beta,c_\beta)}}'_l({\tilde{\lambda}}^+)\le{\overline{L}^{(\mu_\beta,c_\beta)}}'_r({\tilde{\lambda}}^+)\]
   (where  (\ref{end-proof-eq33}) is used to obtain the first inequality), 
   hence by  (\ref{end-proof-eq60}) all the above inequalities are equalities, and the 
   second  equality of the second assertion follows.
    \end{proof-claim 1}

\begin{claim}\label{claim 2} 
When  ${\overline{L}^{(\mu_\beta,c_\beta)}}'_r(\lambda^+)=-\infty$, the     hypothesis  of Theorem \ref{strength-Plachky-Steinebach} holds with the net $(\mu_\beta,c_\beta)$ in place of $(\mu_\alpha,c_\alpha)$; in particular, $\overline{L}^{(\mu_\beta,c_\beta)}(\lambda^+)=\overline{L}(\lambda^+)$ and ${\overline{L}^{(\mu_\beta,c_\beta)}}'_r (\lambda^+)=x$.
\end{claim}

 \begin{proof-claim 2}  
The hypothesis implies  $\lambda=0$ and  $\tilde{\lambda}^{(\mu_\beta,c_\beta)}=0$, hence  $\tilde{\lambda}=0$ by (\ref{end-proof-eq48}); therefore,       $x=-\infty$ by (\ref{end-proof-eq30}), which implies   $\overline{L}(\lambda^+)=0$ by hypothesis of Theorem \ref{strength-Plachky-Steinebach}; this last equality together with 
 (\ref{end-proof-eq20}) yields
  $\overline{L}^{(\mu_\beta,c_\beta)}(\lambda^+)=0$, which proves the claim.
    \end{proof-claim 2}
    
    By Claim \ref{claim 1}  and Claim  \ref{claim 2} we can apply  the first assertion  of Theorem \ref{strength-Plachky-Steinebach} with the net  $(\mu_\beta,c_\beta,x_\beta,y_\beta)$ in place  of  $(\mu_\alpha,c_\alpha,x_\alpha,y_\alpha)$, which gives
  \begin{multline*}
  \limsup_\beta c_\beta\log\mu_\beta([x_\beta,y_\beta[)=\limsup_\beta c_\beta\log\mu_\beta(]x_\beta,y_\beta[)
  =  \overline{L}^{(\mu_\beta,c_\beta)}(\lambda^+)-\lambda{\overline{L}^{(\mu_\beta,c_\beta)}}_r'(\lambda^+)
  \\
  =  \overline{L}(\lambda^+)-\lambda x.
  \end{multline*} 
  Since    the right hand side  of the above last equality  does not depend on $(\mu_\beta,c_\beta,x_\beta,y_\beta)$,  and  
     $(\mu_\beta,c_\beta,x_\beta,y_\beta)$ is an arbitrary subnet of  $(\mu_\alpha,c_\alpha,x_\alpha,y_\alpha)$,  the second  assertion follows.
    \end{proof-theo}

\begin{proof-corollary}
 Let $t\in[\lambda,\lambda+\varepsilon[$ and put ${\tilde{t}}=\sup\{s>t: \overline{L}_{\mid ]t,s]}\ \textnormal{is   affine}\}$ 
 (note that ${\tilde{t}}$ is a maximum since $\overline{L}_{\mid ]\lambda,\lambda+\varepsilon[}$ is continuous).
 Since (by hypothesis) $\sup\left\{\overline{L}'(t): t\in\ ]\lambda,\lambda+\varepsilon[\right\}$ is not a maximum, we have ${\tilde{t}}<\lambda+\varepsilon$. 
For each $\delta>0$,  the map  $\overline{L}_{\mid [{\tilde{t}},{\tilde{t}}+\delta[}$ is not affine 
(otherwise, since $\overline{L}$ is differentiable at ${\tilde{t}}$,  the slope of $\overline{L}_{\mid [{\tilde{t}},{\tilde{t}}+\delta[}$ would be the same as the one  of $\overline{L}_{\mid ]t,{\tilde{t}}]}$, which would contradict  the definition of ${\tilde{t}}$).  Therefore, 
we have   $\overline{L}'(s)>\overline{L}'({\tilde{t}}^+)$ for all $s>{\tilde{t}}$; since  $\overline{L}'(t^+)=\overline{L}'({\tilde{t}}^+)$ (because   $\overline{L}_{\mid ]t,{\tilde{t}}]}$ is affine and $\overline{L}$ is differentiable at ${\tilde{t}}$) we get 
$\overline{L}'(s)>\overline{L}'(t^+)$ for all $s>{\tilde{t}}$. Consequently, by the last assertion of Lemma \ref{existence L on t} $b)$, 
the case $(i)$ of Lemma \ref{existence L on t}$b)$ holds, hence the hypotheses of Theorem \ref{strength-Plachky-Steinebach}  
 holds for $t$ (\ie with $\lambda=t$ in Theorem \ref{strength-Plachky-Steinebach}), and in particular  when $\overline{L}'_r(t^+)>-\infty$.
 If $\overline{L}'_r(t^+)=-\infty$, then $t=\lambda=0$ by Lemma \ref{improper-L}; 
 since in this case $\overline{L}$  is assumed to be right continuous at zero,  the hypotheses of Theorem \ref{strength-Plachky-Steinebach} holds for $t$.  Lemma \ref{lemma-case-differentiable}$a)$  ensures that for each  
  $z\in\left[\overline{L}_r'(\lambda^+),\overline{L}_l'(\lambda+\varepsilon)\right[$  there exists $t_z\in[\lambda,\lambda+\varepsilon[$   such that  $z=\overline{L}_r'(t_z^+)$, 
 hence   part $a)$  follows  applying Theorem \ref{strength-Plachky-Steinebach} for all $t\in[\lambda,\lambda+\varepsilon[$. 
  Part $b)$ follows from Lemma        \ref{lemma-case-differentiable} $b)$.
   \end{proof-corollary}

\section*{Appendix}\label{generic-example-not-diff-llmgf}

\bigskip

Let $\lambda\ge 0$, let  $\varepsilon>0$ and let $f$ be a real-valued strictly convex and continuous  function  on  $[0,\lambda+\varepsilon]$ such that $f(0)=0$. In what follows, we show how  from such a function $f$ it is possible to build 
  a generalized    log-moment generating function $L_f$ with effective domain $[0,\lambda+\varepsilon]$,    such that $L_f(t)$  is a limit for all $t\in\R$,  $L_f$ is not differentiable at $\lambda$,  ${L_f}_{\mid [0,\lambda]}$  is affine,   $\lambda$ is  a limit of a decreasing sequence $(\lambda_i)_{i\ge 1}$ of non-differentiability  points  of ${L_f}_{\mid]\lambda,\lambda+\varepsilon[}$ and  ${L_f}_{\mid [\lambda_i,\lambda_{i-1}]}$  is  affine for all $i\ge 1$ (with $\lambda_0=\lambda+\varepsilon$). Therefore,  the hypotheses of 
   Theorem \ref{strength-Plachky-Steinebach} holds for  $\lambda$,  but the usual version of Plachky-Steinebach  theorem does not  apply (either because  $\overline{L}$ is not differentiable   in a neighbourhood of $\lambda$, or because $\overline{L}$ is not strictly convex  in a neighbourhood of $\lambda$); however,  its conclusion remains true
  (\ie   both conclusions of Theorem \ref{strength-Plachky-Steinebach}  hold);   the case $\lambda=0$ and ${L_f}'_r(0^+)=-\infty$  is included.    
  
 Note that   $\lambda$ is the only point in $[0,\lambda+\varepsilon]$  to which  Theorem  \ref{strength-Plachky-Steinebach}  applies, since  (a) every $t\in[0,\lambda[\ \cup\ ]\lambda,\lambda+\varepsilon[$ fulfils the case $(ii)$ of Lemma \ref{existence L on t}$b)$, hence the excluded case $(ii)$ mentioned in  Section \ref{intro},  and (b)  $L'_r(\lambda+\varepsilon)=+\infty$, \ie $\lambda+\varepsilon$ fulfils the case $(iv)$ of  Lemma \ref{existence L on t}, hence the excluded case $(i)$ mentioned in  Section \ref{intro}.

   First, we extend $f$ to a convex lower semi-continuous function on  $\R$ by putting $f(t)=+\infty$ for all $t\not\in [0,\lambda+\varepsilon]$. Put $\lambda_0=\lambda+\varepsilon$.
  Let  $(\lambda_i)_{i\ge 1}$  be a decreasing sequence in $]\lambda,\lambda+\varepsilon[$   converging to $\lambda$. 
Draw  a line segment $D_i$ between   $(\lambda_i,f(\lambda_i)$ and $(\lambda_{i+1},f(\lambda_{i+1}))$ for all $i\in\N$, and draw a line segment $D$ between  and $(0,0)$ and  $(\lambda,f(\lambda)$. Let $L_f$ be the function whose graph coincides with the graph of $f$ (resp. $D_i$,  $D$) on $]-\infty,0[\ \cup\ ]\lambda_0,+\infty[$ 
 (resp. $[\lambda_{i+1},\lambda_i]$ for all $i\in\N$, $[0,\lambda]$). Clearly, $L_f$ is a convex  function on $\R$,    continuous on its effective domain $[0,\lambda+\varepsilon]$, affine on $[0,\lambda]$ and $[\lambda_{i+1},\lambda_i]$ for all $i\in\N$, 
 and  fulfils $L_f(0)=0$.
  The  strict convexity of $f$  ensures 
    for each $i\ge 1$ the existence of some  $t_i\in\ ]\lambda_{i+1},\lambda_i[$ such that 
 \begin{equation}\label{appendix-eq20}
 f'_r(\lambda)<f'_r(t_{i+1})<{L_f}'_r(\lambda_{i+1})<f'_r(t_i)<{L_f}'_r(\lambda_i),
 \end{equation}
 hence 
   \begin{equation}\label{appendix-eq40}
   {L_f}'_r(\lambda)=f'_r(\lambda)=\lim_i {L_f}'_r(\lambda_i)
    \end{equation}
    and  ${L_f}'_l(\lambda)<{L_f}'_r(\lambda)$ when ${L_f}'_r(\lambda)>-\infty$. 
  Therefore,  $L_f$ is not differentiable at $\lambda$, and not differentiable at $\lambda_i$ for all $i\in\N$, but (\ref{appendix-eq20}) and (\ref{appendix-eq40}) show that 
  $L_f$     fulfils the hypotheses of Theorem \ref{strength-Plachky-Steinebach}.

It remains to show that $L_f$ is  a generalized  log-moment generating function such that $L_f(t)$ is a limit for all $t\in\R$.   Let $\{z_k:k\ge 1\}$ be a countable set dense in the effective domain $\textnormal{dom } L_f^*$ of $L_f^*$,   and put 
\[\forall n\ge 1,\ \ \ \ \ \ \ \mu_{n,f}=\frac{\sum_{k=1}^n e^{-n L_f^*(z_k)}}{\sum_{k=1}^n e^{-n L_f^*(z_k)}}\delta_{z_k}.\]
Since   $L_f$ is lower semi-continuous, we have  $L_f=L_f^{**}$ (\cite{Rockafellar-70}, Theorem 23.5) \ie
\begin{equation}\label{construction-LDP-sequence-with-prescribed-llmgf-eq10}
\forall t\in\R,\ \ \ \ \ \ \ \ L_f(t)=\sup_{z\in\R}\{tz-L_f^*(z)\}=\sup_{z\in\textnormal{dom }L_f^*}\{tz-L_f^*(z)\},
\end{equation}
hence
\begin{multline*}
\forall n\ge 1,\ \ \ \ \ \ \ \ \max_{k=1}^n\{tz_k - L_f^*(z_k)\}\le n^{-1}\log\sum_{k=1}^n e^{n(tz_k - L_f^*(z_k))}
\\
\le n^{-1}\log (n \max\{e^{n(tz_k - L_f^*(z_k))}\})
 =n^{-1}\log n  + \max_{k=1}^n\{tz_k - L_f^*(z_k)\}\le n^{-1}\log n  + L_f(t)
 \end{multline*}
  and letting $n\rightarrow+\infty$,
     \begin{multline*}
\sup_{k=1}^\infty\{tz_k - L_f^*(z_k)\}\le\liminf_n n^{-1}\log\sum_{k=1}^n e^{n(tz_k - L_f^*(z_k))}
\\
\le\limsup_n n^{-1}\log\sum_{k=1}^n e^{n(tz_k - L_f^*(z_k))}\le L_f(t).
\end{multline*}
 Since the set $\{z_k:k\ge 1\}$ is dense in $\textnormal{dom\ }L_f^*$ and 
 $L_f^*$ is continuous on $\textnormal{dom\ }L_f^*$, we  get for each $t\in\R$ and for each $z\in\textnormal{dom } L_f^*$,
  \[
tz - L_f^*(z)\le\liminf_n n^{-1}\log\sum_{k=1}^n e^{n(tz_k - L_f^*(z_k))}\le
\limsup_n n^{-1}\log\sum_{k=1}^n e^{n(tz_k - L_f^*(z_k))}
\le L_f(t).
\]
hence 
 \begin{multline*}
 \forall t\in\R,\ \ \ \ \ \ \sup_{z\in\textnormal{dom }L^*}\{tz-L^*(z)\le\liminf_n n^{-1}\log\sum_{k=1}^n e^{n(tz_k - L_f^*(z_k))}
 \\
 \le\limsup_n n^{-1}\log\sum_{k=1}^n e^{n(tz_k - L_f^*(z_k))}
\le L_f(t),
 \end{multline*}
which together with (\ref{construction-LDP-sequence-with-prescribed-llmgf-eq10}) gives
\begin{equation}\label{construction-LDP-sequence-with-prescribed-llmgf-eq20}
\forall t\in\R,\ \ \ \ \ \ \lim_n n^{-1}\log\sum_{k=1}^n e^{n(tz_k - L_f^*(z_k))}= L_f(t).
\end{equation}
Since 
\[n^{-1}\log\int_\R e^{ntz}\mu_{n,f}(dz)=n^{-1}\left(\log\sum_{k=1}^n e^{n(tz_k - L_f^*(z_k))}-\log\sum_{k=1}^n e^{-nL_f^*(z_k)}\right)\]
for all $(t,n)\in\R\times\N\setminus\{0\}$, 
it follows from (\ref{construction-LDP-sequence-with-prescribed-llmgf-eq20}) that  the generalized  log-moment generating function associated with $(\mu_{n,f},n^{-1})$ is a limit for all  $t\in\R$,  and coincides with   $L_f$.
Consequently,   both conclusions of Theorem \ref{strength-Plachky-Steinebach} hold, \ie 
for every sequence $(x_n)$ converging to ${L_f}_r'(\lambda)$,  and for  every sequence $(y_n)$ fulfilling $\liminf_n y_n>{L_f}_r'(\lambda)$, we have 
\begin{multline*}
\lim_n n^{-1}\log\mu_{n,f}([x_n,y_n])=\lim_n n^{-1}\log\mu_{n,f}(]x_n,y_n[)=L_f(\lambda)-\lambda {L_f}_r'(\lambda)
       \\
       = \left\{\begin{array}{ll}
  -{L_f}^*({L_f}_r'(\lambda))=\inf_{t>0}\{{L_f}(t)-t {L_f}_r'(\lambda)\}  & if\ \lambda>0\ or\   (\lambda=0\ and\ {L_f}_r'(0)>-\infty)
  \\
 \lim_n {L_f}^*(x_n) =0  &  if\  \lambda=0\ and\ {L_f}_r'(0)=-\infty.
      \end{array}
    \right.
\end{multline*}

\end{document}